\documentclass[11pt]{article}
\usepackage{etex}
\usepackage{tikz}
\usepackage{tikz-cd}
\usepackage[affil-it]{authblk}
\usepackage{fullpage}

\RequirePackage{amsmath,amsfonts,amssymb}
\RequirePackage{mathrsfs}
\RequirePackage{graphicx}
\RequirePackage{epic,eepic}
\RequirePackage{color}

\let\dim\relax

\DeclareMathOperator{\dim}{dim}
\DeclareMathOperator{\tr}{tr}

\let\im\relax

\DeclareMathOperator{\im}{im}

\DeclareMathOperator{\coker}{coker}
\DeclareMathOperator{\Ext}{Ext}
\DeclareMathOperator{\Tor}{Tor}
\DeclareMathOperator{\Hom}{Hom}
\DeclareMathOperator{\id}{id}
\DeclareMathOperator{\Aut}{Aut}

\newcommand{\gen}[1]{\left\langle#1\right\rangle} % generating set
\newcommand{\then}{\Rightarrow}

\renewcommand{\iff}{\Leftrightarrow}
\newcommand{\Iff}{\Longleftrightarrow}

\newcommand{\zz}{\mathbb{Z}}

\let\ov\overline
\newcommand{\ova}{\ov{a}}
\DeclareMathOperator{\cd}{cd}

\newtheorem{lemma}{Lemma}[section]
\newtheorem{theorem}[lemma]{Theorem}
\newtheorem{proposition}[lemma]{Proposition}

\newenvironment{proof}{\medskip\noindent\textit{Proof: \/}}
               {\begin{flushright}\rule{2mm}{2mm}\end{flushright}\par\medskip}
\newenvironment{remark}{\medskip\noindent\textbf{Remark: }}
               {\par\bigskip}
\newenvironment{example}{\medskip\noindent\textbf{Example: }}
               {\par\bigskip}

\begin{document}

\title{The cohomology ring of the sapphires that admit the Sol geometry}
\author{S\'ergio Tadao Martins\footnote{Electronic address: \texttt{sergio.tadao.martins@ufsc.br}. The author has received financial support from FAPESP, processes 2013/07510-4 and 2013/21394-7}}
\affil{Department of Mathematics, Federal University of Santa Catarina}
\author{Daciberg Lima Gon\c calves\footnote{Electronic address: \texttt{dlgoncal@ime.usp.br}}}
\affil{Institute of Mathematics and Statistics, University of S\~ao Paulo}
\date{\today}
\maketitle

\begin{abstract}
{\let\thefootnote\relax}
{\let\thefootnote\relax}
Let $G$ be the fundamental group of a sapphire that admits the
\emph{Sol\/} geometry and is not a torus bundle. We determine a finite
free resolution of $\zz$ over $\zz G$ and calculate a partial diagonal
approximation for this resolution. We also compute the cohomology
rings $H^*(G;A)$ for $A=\zz$ and $A=\zz_p$ for an odd prime $p$, and
indicate how to compute the groups $H^*(G;A)$ and the multiplicative
structure given by the cup product for any system of coefficients $A$.

%\bigskip
%\noindent{\it Keywords: } torus semi-bundles; sapphires; Sol
%$3$-manifolds; finite free resolution; diagonal approximation; cohomology ring
%
%\bigskip
%\noindent{Mathematics Subject Classification 2010: } primary: 20J06;
%secondary: 57M05; 57R19
\end{abstract}

\section{Introduction}

In this paper we consider the problem of calculating the cohomology
ring of $3$-manifolds that admit the {\it Sol\/} geometry. The
$3$-manifolds that admit this geometry are either {\it torus bundles}
or {\it torus semi-bundles}. In~\cite{Martins2013} we considered the
problem of determining the cohomology ring of the torus bundles, so
we now focus on the cohomology of torus semi-bundles.

A torus semi-bundle (also called a {\it sapphire}) is an orientable
$3$-manifold obtained from two twisted $I$-bundles $K_1I$ and $K_2I$
over the Klein bottle by identifying their boundaries, as described
in~\cite{Morimoto}. Using the notation from that article, let $X =
K(r,s,t,u)$ be the sapphire obtained when we identify the boundaries
by a homeomorphism $h\colon \partial K_2I \to \partial K_1I$ such that
the induced isomorphism $h_*\colon \pi_1(\partial K_2I) \to
\pi_1(\partial K_1I)$ is given by $h_*(\alpha_2) = r\alpha_1+s\beta_1$
and $h_*(\beta_2) = t\alpha_1+u\beta_1$, where the meanings of the
generators $\alpha_i$ and $\beta_i$ for $i=1,\,2$ are explained
in~\cite{Morimoto}. If we let $G=\pi_1(X)$, then the space $X$ is a
$K(G,1)$ space, so its cohomology coincides with the cohomology of
$G$.

Morimoto proved in~\cite{Morimoto} that $X$ is a torus bundle over
$S^1$ if, and only if, $t=0$. In~\cite{Martins2013}, a finite free
resolution of $G$ was constructed and also a diagonal approximation
for that resolution was given when $X$ is a torus bundle, and with
those ingredients the cohomology rings $H^*(G;\zz)$ and $H^*(G;\zz_p)$
for a prime $p$ were computed. Moreover, the results obtained
in~\cite{Martins2013} also allow us to calculate the cup product
$H^p(G;M) \otimes H^q(G;N) \to H^{p+q}(G;M\otimes N)$ for nontrivial
coefficients $M$ and $N$. As we are interested in the cohomology ring
of $G$, we can exclude the case where $X$ is a torus bundle over
$S^1$, so from now on we assume $t\ne 0$. Moreover, if the sapphire
$X=K(r,s,t,u)$ admits the $Sol$ geometry and is not a torus bundle,
then we actually have $rstu\ne 0$ (see~\cite{SunWangWu}), so that's
what we assume from now on.

The rest of the paper is organized as follows: in
Section~\ref{section:G} we exhibit some properties of the group $G$,
including some normal forms for its elements, and of the group ring
$\zz G$. In Section~\ref{section:freeresolution}, we construct a
finite free resolution $F$ of $\zz$ over $\zz G$, determine a partial
diagonal approximation for $F$ and discuss how to compute cup products
for arbitrary systems of coefficients. In
Section~\ref{section:examples} we show some examples of computations
for the cohomology of $G$ with twisted integer coefficients and also
determine the ring $H^*(G;\zz_p)$ when $p$ is an odd prime.

\section{The group $G$ and the group ring $\zz G$}\label{section:G}

Morimoto showed~\cite{Morimoto} that the fundamental group $G$ of the
torus semi-bundle $X = K(r,s,t,u)$ is given by the presentation
\begin{equation}\label{eq:Gpresentation}
G = \gen{a_1,b_1,a_2 \mid a_1b_1a_1^{-1} = b_1^{-1}, \: a_2^2 = a_1^{2r}b_1^s, \: a_2a_1^{2t}b_1^ua_2^{-1} = b_1^{-u}a_1^{-2t}}
\end{equation}
and from that he also got
\[
H_1(G;\zz) = \begin{cases}
\zz_{4t}\oplus \zz_2\oplus \zz_2, & \text{if $s$ is even,} \\
\zz_{4t}\oplus \zz_4, & \text{if $s$ is odd.} \\
\end{cases}
\]

Since we are assuming that $t\ne 0$, $H_1(G;\zz)$ is finite and it
follows from the universal coefficient theorem for cohomology that
$H^1(G;\zz) = 0$. The computation of $H^2(G;\zz)$ follows from
Poincar\'e duality (remembering that $X$ is orientable), which also
gives us $H^3(G;\zz) = \zz$. We state this result formally in the next
theorem for future reference.

\begin{theorem}\label{theorem:integralcohomology}
The cohomology groups $H^*(G;\zz)$ are given by
\begin{align*}
H^0(G;\zz) &= \zz, \\
H^1(G;\zz) &= 0, \\
H^2(G;\zz) &= \begin{cases}
\zz_{4t}\oplus \zz_2\oplus \zz_2, & \text{if $s$ is even,} \\
\zz_{4t}\oplus \zz_4, & \text{if $s$ is odd,} \\
\end{cases}\\
H^3(G;\zz) &= \zz.
\end{align*}
\end{theorem}

The above theorem tells us everything about the cohomology of $G$ with
trivial integer coefficients, since $H^1(G;\zz) = 0$ and $\cd(G)=3$
imply that all the potentially nontrivial cup products are actually
null. For other coefficients, however, we need to know more about $G$.

First of all we present two exact sequences featuring $G$. The reader
can consult~\cite{GoncalvesWong} for more details. If we let $x =
a_1^2$ and $y = b_1$, then $xy=yx$ and the subgroup $N$ of $G$
generated by $x$ and $y$ is normal in $G$ and $Q = G/N =
\gen{\ova_1,\ova_2 \mid \ova_1^2 = \ova_2^2 = 1} \cong \zz_2*\zz_2$,
so $G$ fits in an exact sequence of groups
\[
\begin{tikzcd}
\zz\oplus\zz \arrow[tail]{r} & G \arrow[twoheadrightarrow]{r} & \zz_2 *\zz_2.
\end{tikzcd}
\]
The exact sequence above implies that every element $g\in G$ can be
uniquely written as $g = wx^iy^j$, where $i$ and $j$ are integers and
$w$ is a (possibly empty) word in the alphabet $\{a_1,a_2\}$ with
alternating letters. This can be seen as a ``normal form'' for the
elements of $G$. This normal form is used to define the map $s_0$ that
appears just before Theorem~\ref{theorem:Gdiagonal}.

If we let $v = a_1^{-1}a_2$, then the subgroup $L$ of $G$ generated by
$x$, $y$ and $v$ is normal of index $2$, for $L$ is the kernel of
$\eta\colon G\to \zz_2 = \gen{\ova_1 \mid \ova_1^2 = 1}$ defined by
$\eta(a_1) = \ova_1$, $\eta(a_2) = \ova_1$, $\eta(b_1) = 1$. It
follows from the presentation~(\ref{eq:Gpresentation}) that
\begin{align}
a_2xa_2^{-1} &= x^{ru+st}y^{2su}, \label{eq:a2cjgx} \\
a_2ya_2^{-1} &= x^{-2rt}y^{-ru-st}, \label{eq:a2cjgy} \\
\intertext{hence}
vxv^{-1} &= x^{ru+st}y^{-2su}, \label{eq:vcjgx} \\
vyv^{-1} &= x^{-2rt}y^{ru+st}. \label{eq:vcjgy}
\end{align}
Therefore we have $L = \gen{x,y\mid xy=yx}\rtimes_\theta\gen{v} \cong
(\zz\oplus\zz)\rtimes_\theta\zz$, where $\theta\colon \zz \to
\Aut(\zz\oplus\zz)$ is given by the matrix
\[
\theta(1) = \begin{pmatrix}
ru+st & -2rt \\
-2su & ru+st
\end{pmatrix},
\]
and so there is another exact sequence of groups
\[
\begin{tikzcd}
(\zz\oplus\zz)\rtimes_\theta\zz \arrow[tail]{r} & G
  \arrow[twoheadrightarrow]{r} & \zz_2.
\end{tikzcd}
\]

The matrix $\theta(1)$ satisfies $\det\theta(1) = 1$ and
$|\tr(\theta(1))| > 2$, which implies that $\theta(1)$ doesn't admit
$\pm 1$ as an eigenvalue.

The last exact sequence implies that every element of $G$ can be
uniquely written as $x^my^nv^k$ or $x^my^nv^ka_1$ for $m$, $n$, $k\in
\zz$, which is a second normal form for the elements of $G$. A third
normal form is as follows: each $g\in G$ can be uniquely written
either as $x^my^nv^k$ or $x^my^nv^ka_2$. To prove that, we note first
that
\begin{align*}
a_1va_1^{-1} &= a_1(va_1^{-2})a_1 \\
&= a_1(vx^{-1})a_1 \\
&= a_1(x^{-ru-st}y^{2su}v)a_1 \\
&= (a_1x^{-ru-st}y^{2su})va_1 \\
&= x^{-ru-st}y^{-2su}(a_1v)a_1 \\
&= x^{-ru-st}y^{-2su}a_2a_1 \\
&= x^{-ru-st}y^{-2su}(x^ry^s)a_2^{-1}a_1 \\
&= x^{r-ru-st}y^{s-2su}v^{-1}.
\end{align*}
Now, if we also denote by $\theta$ the map $N \to N$ given by
$\theta(x) = x^{ru+st}y^{-2su}$, $\theta(y) = x^{-2rt}y^{ru+st}$, an
element of the group $G$ of the form $g = x^my^nv^ka_2$ can be written
as
\begin{align*}
g &= x^my^nv^k(a_1v) \\
&= x^my^nv^k(x^{r-ru-st}y^{s-2su}v^{-1}a_1) \\
&= x^my^n\theta^k(x^{r-ru-st}y^{s-2su})v^{k-1}a_1.
\end{align*}
Our claim now follows from our known previous normal form for the
elements of $G$. Finally, we also note that there is a fourth normal
form for the elements of $G$: each $g\in G$ can be uniquely written as
$v^kx^my^n$ or $v^kx^my^na_2$. This is consequence of our third normal
form and of the fact that the matrix $\theta(1)$ is invertible
($\det\theta(1) = (ru-st)^2 = 1$). This last normal form will be used
in the proof of Theorem~\ref{theorem:Gresolution}

\bigskip
Now we recall one property of the group ring $\zz G$ that we will need
later, and that is the fact that $\zz G$ has no nontrivial zero
divisors. This follows from the main result proved
in~\cite{FarkasSnider}.

\section{Free resolution of $\zz$ over $\zz G$}\label{section:freeresolution}

In this section we determine a finite free resolution $F$ of $\zz$
over $\zz G$, a partial diagonal approximation $\Delta$ for the
resolution $F$, and discuss how to compute the cohomology ring of $G$
for arbitrary coefficients.

\begin{theorem}\label{theorem:Gresolution}
A free resolution of $\zz$ over $\zz G$ is given by
\[
\begin{tikzcd}
0 \arrow{r} & F_3 \arrow{r} & F_2 \arrow{r}{d_2} & F_1 \arrow{r}{d_1} & F_0 \arrow{r}{\varepsilon} & \zz \arrow{r} & 0,
\end{tikzcd}
\]
where $F_0 = \zz G$, $F_1 = \zz G^3$, $F_2 = \zz G^3$ and $F_3 =\zz
G$. More precisely, if we call $\alpha_1$, $\beta_1$ and $\alpha_2$
the generators of $F_1$ and $\rho_1$, $\rho_2$ and $\rho_3$ the
generators of $F_2$, then the maps $d_1$, $d_2$, and $\varepsilon$ are
given by
\begin{align*}
\varepsilon(1) &= 1, \\
d_1(\alpha_1) &= a_1-1, \\
d_1(\beta_1) &= b_1-1, \\
d_1(\alpha_2) &= a_2-1, \\
d_2(\rho_1) &= (1-b_1^{-1})\alpha_1 + (a_1 + b_1^{-1})\beta_1, \\
d_2(\rho_2) &= \dfrac{\partial a_1^{2r}}{\partial a_1}\alpha_1 + a_1^{2r}\dfrac{\partial b_1^s}{\partial b_1}\beta_1 + (-a_2-1)\alpha_2, \\
d_2(\rho_3) &= \left(a_2\dfrac{\partial a_1^{2t}}{\partial a_1} + a_1^{-2t}b_1^{-u}\dfrac{\partial a_1^{2t}}{\partial a_1}\right)\alpha_1 + \\
&\phantom{=}+ \left(a_2a_1^{2t}\dfrac{\partial b_1^u}{\partial b_1} + b_1^{-u}\dfrac{\partial b_1^u}{\partial b_1}\right)\beta_1 + (1 - a_1^{-2t}b_1^{-u})\alpha_2,
\end{align*}
where the partial derivatives are the Fox derivatives.
\end{theorem}
\begin{proof}
Let $r_1 = a_1b_1a_1^{-1}b_1$, $r_2 = a_1^{2r}b_1^sa_2^{-2}$ and $r_3
=a_2a_1^{2t}b_1^ua_2^{-1}a_1^{2t}b_1^u$. From~(\ref{eq:Gpresentation}),
we can write $G = \gen{a_1,b_1,a_2 \mid r_1,\:r_2,\:r_3}$.
Following~\cite[chapter~2,~section~3]{LyndonSchupp77},
the sequence
\[
\begin{tikzcd}
F_2 \arrow{r}{d_2} & F_1 \arrow{r}{d_1} & F_0 \arrow{r}{\varepsilon} & \zz \arrow{r} & 0,
\end{tikzcd}
\]
is exact. But $\cd(G)=3$ implies that $\ker(d_2)$ is projective
(see~\cite[Lemma~VIII.2.1]{Brown82}). We are going to prove that
$\ker(d_2)$ is actually free and isomorphic to $\zz G$.

The element $X\rho_1 + Y\rho_2 + Z\rho_3$ belongs to $\ker(d_2)$ if,
and only if,
\begin{equation}\label{eq:systemM3}
\left|\begin{array}{l}
X(1-b_1^{-1}) + Y\dfrac{\partial a_1^{2r}}{\partial a_1} + Z\left(a_2\dfrac{\partial a_1^{2t}}{\partial a_1} + a_1^{-2t}b_1^{-u}\dfrac{\partial a_1^{2t}}{\partial a_1}\right) = 0 \\
X(a_1+b_1^{-1}) + Ya_1^{2r}\dfrac{\partial b_1^s}{\partial b_1} + Z\left(a_2a_1^{2t}\dfrac{\partial b_1^u}{\partial b_1} + b_1^{-u}\dfrac{\partial b_1^u}{\partial b_1}\right) = 0 \\
Y(-a_2-1) + Z(1-a_1^{-2t}b_1^{-u}) = 0.
\end{array}\right.
\end{equation}
Let's focus on the third equation $Y(a_2+1) =
Z(1-x^{-t}y^{-u})$. Given an element $g\in G$, write it in the normal
form $g = v^kx^my^na_2^{\ell}$, where $\ell$ is $0$ or $1$. We have
\begin{align*}
&v^kx^my^n(a_2+1) = v^kx^my^na_2 + v^kx^my^n, \\
&v^kx^my^na_2(a_2+1) = v^kx^{m+r}y^{n+s} + v^kx^my^na_2, \\
&v^kx^my^n(1-x^{-t}y^{-u}) = v^kx^my^n -v^kx^{m-t}y^{n-u}, \\
&v^kx^my^na_2(1-x^{-t}y^{-u}) = v^kx^my^na_2 - v^kx^{m+t}y^{n+u}a_2.
\end{align*}
The important point to notice is that in the four equations above, the
exponent of $v$ is the same on the left side and on both terms of the
right side. This motivates the following definition: given $W =
\sum_{k,m,n} \phi_{k,m,n}v^kx^my^n + \psi_{k,m,n}v^kx^my^na_2 \in \zz
G$ and $\ell\in\zz$, the {\it $\ell$-component\/} of $W$ is the sum of
all the terms of $W$ where the exponent of $v$ is equal to $\ell$:
\[
\text{$\ell$-component of $W$} = \sum_{m,n\in\zz} \phi_{\ell,m,n}v^\ell x^my^n
+ \psi_{\ell,m,n}v^\ell x^my^na_2.
\]
Therefore, if $Y(a_2+1) = Z(1-x^{-t}y^{-u})$, then this equality
remains valid if we substitute $Y$ and $Z$ by their respective
$\ell$-components. So let
\begin{gather*}
\sum_{m,n} \alpha_{m,n}v^\ell x^my^n + \beta_{m,n}v^\ell x^my^na_2 \\
\intertext{and}
\sum_{m,n} \gamma_{m,n}v^\ell x^my^n + \delta_{m,n}v^\ell x^my^na_2
\end{gather*}
be the $\ell$-components of $Y$ and $Z$, respectively. Now $Y(a_2+1) =
Z(1-x^{-t}y^{-u})$ implies
\begin{gather*}
\sum_{m,n} (\alpha_{m,n}+\beta_{(m-r),(n-s)})v^\ell x^my^n + (\alpha_{m,n}+\beta_{m,n})v^\ell x^my^na_2 = \\
=\sum_{m,n} (\gamma_{m,n}-\gamma_{(m+t),(n+u)})v^\ell x^my^n + (\delta_{m,n}-\delta_{(m-t),(n-u)})v^\ell x^my^na_2.
\end{gather*}
From the above equality it follows that
\begin{align*}
(\alpha_{m,n}+\beta_{m,n}) - (\alpha_{m,n}+\beta_{(m-r),(n-s)}) &= (\delta_{m,n}-\delta_{(m-t),(n-u)}) - (\gamma_{m,n}-\gamma_{(m+t),(n+u)}) \Iff\\
\beta_{m,n} - \beta_{(m-r),(n-s)} &= (\delta_{m,n}-\delta_{(m-t),(n-u)}) - (\gamma_{m,n}-\gamma_{(m+t),(n+u)}),
\end{align*}
and therefore
\[
\sum_{k\in\zz} \beta_{(m+kt),(n+ku)} - \beta_{(m-r+kt),(n-s+ku)} = 0.
\]
The sum on the left side makes sense because only finitely many of the
coefficients $\beta_{m,n}$, $\gamma_{m,n}$ and $\delta_{m,n}$ are
distinct from zero. If we define $B(m,n) = \sum_{k\in\zz}
\beta_{(m+kt),(n+ku)}$, the above equality means that
\[
B(m,n) = B(m-r,n-s) = B(m-2r,n-2s) = \cdots = B(m-kr,n-ks) = \cdots,
\]
that is, $B(m,n)=B(m-kr,m-ks)$ for all $k\in\zz$. Since only finitely
many $\beta_{m,n}$ are distinct from zero and $(r,s)$ and $(t,u)$ are
linearly independent, we actually have $B(m,n) = 0$ for all $m$,
$n$. A similar argument beginning with the identity
\[
\alpha_{m,n}+\beta_{m,n} = \delta_{m,n}-\delta_{(m-t),(n-u)}
\]
shows that $A(m,n) = \sum_{k\in\zz}\alpha_{(m+kt,n+ku)} = 0$ for all
$m$, $n$.  Now we notice that the conditions $A(m,n)=0$ and $B(m,n)=0$
for all $m$, $n\in\zz$ imply that the $\ell$-component of $Y$ can be
factored as $Y'(1-x^ty^u)$. The same then can be said for $Y$, so
there is a $W\in \zz G$ such that $Y = W(1-x^ty^u)$ and
\begin{gather*}
Y(a_2+1) = Z(1-x^{-t}y^{-u}) \Iff\\
W(1-x^ty^u)(a_2+1) = Z(1-x^{-t}y^{-u}) \Iff\\
W(a_2-x^ty^u)(1-x^{-t}y^{-u}) = Z(1-x^{-t}y^{-u}) \Iff\\
W(a_2-x^ty^u) = Z.
\end{gather*}
Here we have used the fact that $\zz G$ has no nontrivial zero
divisors, so the cancellation law holds. Given $Y=W(1-x^ty^u)$ and
$Z=W(a_2-x^ty^u)$, let's show that there is exactly one $X$ that
satisfies the first two equations of~(\ref{eq:systemM3}). It is clear
that there is at most one such $X$. Let $Y_0 = 1-x^ty^u$ and $Z_0 =
a_2-x^ty^u$, and suppose that $r>0$, $t<0$. In this case,
\begin{equation}\label{eq:X0}
X_0 = \left(\sum_{k=t}^{r+t-1}a_1^{2k}\frac{\partial b_1^u}{\partial b_1} +
\sum_{k=r+t}^{r-1}a_1^{2k}\frac{\partial b_1^s}{\partial b_1} + 
\sum_{k=t+1}^{r+t}a_1^{2k-1}\frac{\partial b_1^{-u}}{\partial b_1} +
\sum_{k=r+t+1}^ra_1^{2k-1}\frac{\partial b_1^{-s}}{\partial b_1}\right)b_1
\end{equation}
is such that $X=X_0$, $Y=Y_0$ and $Z=Z_0$ is a solution
of~(\ref{eq:systemM3}). In all the other cases, we can similarly
determine $X_0$, but Theorem 1 of~\cite{Morimoto} shows that there is
no loss of generality in assuming $r>0$ and $t<0$. Hence $\ker(d_2)$,
as a $\zz G$-module, is free, isomorphic to $\zz G$ and generated by
the element $X_0\rho_1 + Y_0\rho_2 + Z_0\rho_3$. With this notation,
the map $d_3\colon F_3\to F_2$ is defined by $d_3(1) = X_0\rho_1 +
Y_0\rho_2 + Z_0\rho_3$.
\end{proof}

We now present a partial diagonal approximation for the free
resolution of Theorem~\ref{theorem:Gresolution}. In order to do that,
the next two propositions proved by Handel (see~\cite{Handel}) tell us
how we can calculate a diagonal approximation for a given free
resolution, provided that we can find a contracting homotopy for it.

\begin{proposition}[Handel]\label{proposition:handelcontraction}
For a group $\Gamma$, let
\[
\begin{tikzcd}
  \cdots \arrow{r} & C_n \arrow{r} & \cdots \arrow{r} & C_1 \arrow{r} & C_0 \arrow{r}{\varepsilon} \arrow{r} & \zz \arrow{r} & 0
\end{tikzcd}
\]
be a free resolution of $\zz$ over $\zz \Gamma$. If $s$ is a
contracting homotopy for the resolution $C$, then a contracting
homotopy $\tilde{s}$ for the free resolution $C\otimes C$ of $\zz$
over $\zz \Gamma$ is given by
\begin{align*}
\tilde{s}_{-1}\colon &\zz \to C_0\otimes C_0 \\
&\tilde{s}_{-1}(1) = s_{-1}(1)\otimes s_{-1}(1), \\
\tilde{s}_n\colon &(C\otimes C)_n \to (C\otimes C)_{n+1} \\
& \tilde{s}_n(u_i\otimes v_{n-i}) = s_i(u_i)\otimes v_{n-i} + s_{-1}\varepsilon(u_i)\otimes s_{n-i}(v_{n-i}) \quad\text{if $n\ge 0$},
\end{align*}
where $s_{-1}\varepsilon\colon C_0\to C_0$ is extended to
$s_{-1}\varepsilon = \{(s_{-1}\varepsilon)_n \colon C_n \to C_n\}$ in
such a way that $(s_{-1}\varepsilon)_n = 0$ for $n\ge 1$.
\end{proposition}

\begin{proposition}[Handel]\label{proposition:handeldiagonal}
For a group $\Gamma$, let
\[
\begin{tikzcd}
  \cdots \arrow{r} & C_n \arrow{r}{d_n} & \cdots \arrow{r} & C_1 \arrow{r}{d_1} & C_0 \arrow{r}{\varepsilon} \arrow{r} & \zz \arrow{r} & 0
\end{tikzcd}
\]
be a free resolution of $\zz$ over $\zz \Gamma$. Let $B_n$ be a $\zz
\Gamma$-basis for $C_n$ for each $n\ge 0$ such that $\varepsilon(b)=1$
for all $b\in B_0$, and let $s$ be a contracting homotopy for this
resolution $C$. If $\tilde{s}$ is the contracting homotopy for the
resolution $C\otimes C$ given by
Proposition~\ref{proposition:handelcontraction}, then a diagonal
approximation $\Delta\colon C\to C\otimes C$ can be defined in the
following way: for each $n\ge 0$, the map $\Delta_n\colon C_n \to
(C\otimes C)_n$ is given in each element $\rho \in B_n\subset C_n$ by
\[
\begin{array}{l}
\Delta_0 = s_{-1}\varepsilon\otimes s_{-1}\varepsilon,\\
\Delta_n(\rho) = \tilde{s}_{n-1}\Delta_{n-1}d_n(\rho) \quad\text{if $n\ge 1$}.
\end{array}
\]
\end{proposition}

\begin{remark}
Tomoda and Zvengrowski have used the propositions above to calculate
the cohomology rings of some $4$-periodic
groups~\cite{TomodaPeter2008}.
\end{remark}

\begin{remark}
Using the same notation of the previous propositions and using
$[\phantom{M}]$ to denote cohomology classes, if $u \in \Hom_{\zz
  \Gamma}(C_p,M)$ and $v\in \Hom_{\zz \Gamma}(C_q,N)$, where $M$ and
$N$ denote arbitrary $\zz \Gamma$-modules, the cup product $[u]\smile
[v] \in H^{p+q}(\Gamma; M\otimes N)$ is the cohomology class of the
homomorphism $(u\otimes v)\circ\Delta_{p,q}$, where $\Delta_{p,q}$
denotes the composition of $\Delta_{p+q}$ with the projection
$\pi_{p,q}\colon (C\otimes C)_{p+q} \to C_p\otimes C_q$. Hence, for
the calculation of
\[
H^p(\Gamma; M) \otimes H^q(\Gamma; N) \stackrel{\smile}{\to} H^{p+q}(\Gamma; M\otimes N),
\]
we don't need to discover the map $\Delta_{p+q}$, but only
$\Delta_{p,q}$. We'll use this fact later.
\end{remark}

Notice that, if we define the maps of abelian groups $s_{-1}\colon\zz
\to F_0$ by $s_{-1}(1) = 1$ and $s_0\colon F_0 \to F_1$ by $s_0(g) =
\dfrac{\partial g}{\partial a_1}\alpha_1 + \dfrac{\partial g}{\partial
  b_1}\beta_1 + \dfrac{\partial g}{\partial a_2}\alpha_2$ whenever
$g\in G$ is written in the first normal form $g = wx^iy^j$ mentioned
in Section~\ref{section:G}, then $d_1s_0+s_{-1}\varepsilon =
\id_{F_0}$. Now a simple application of the
Propositions~\ref{proposition:handelcontraction}
and~\ref{proposition:handeldiagonal} determines a partial diagonal
approximation for our projective resolution of $\zz$ over $\zz G$.

\begin{theorem}\label{theorem:Gdiagonal}
A partial diagonal approximation $\Delta\colon F\to (F\otimes F)$ for
the free resolution of Theorem~\ref{theorem:Gresolution} is given by
{\allowdisplaybreaks
\begin{align*}
&\Delta_1 \colon F_1 \to (F\otimes F)_1\\
&\quad\Delta_1(\alpha_1) = \alpha_1\otimes a_1 + 1\otimes \alpha_1, \\
&\quad\Delta_1(\beta_1) = \beta_1\otimes b_1 + 1\otimes \beta_1, \\
&\quad\Delta_1(\alpha_2) = \alpha_2\otimes a_2 + 1\otimes \alpha_2, \\
&\Delta_{1,1} \colon F_2 \to (F_1\otimes F_1) \\
&\quad\Delta_{1,1}(\rho_1) = b_1^{-1}\beta_1\otimes b_1^{-1}\alpha_1 + \alpha_1\otimes a_1\beta_1 - b_1^{-1}\beta_1\otimes \beta_1,\\
&\quad\Delta_{1,1}(\rho_2) = s_0\left(\frac{\partial a_1^{2r}}{\partial a_1}\right)\otimes\alpha_1 + s_0\left(a_1^{2r}\frac{\partial b_1^s}{\partial b_1}\right)\otimes \beta_1 - \alpha_2\otimes\alpha_2,\\
&\quad\Delta_{1,1}(\rho_3) = s_0\left(a_2\frac{\partial a_1^{2t}}{\partial a_1} + a_1^{-2t}b_1^{-u}\frac{\partial a_1^{2t}}{\partial a_1}\right)\otimes\alpha_1 + \\
&\quad\phantom{\Delta_{1,1}(\rho_3) =} + s_0\left(a_2a_1^{2t}\frac{\partial b_1^u}{\partial b_1} + b_1^{-u}\frac{\partial b_1^u}{\partial b_1}\right)\otimes\beta_1 - s_0(a_1^{-2t}b_1^{-u})\otimes\alpha_2.
\end{align*}
}
\end{theorem}

\begin{remark}
In the above expressions, the elements on which we need to compute
$s_0$ are already in the normal form $wx^iy^j$ or can be written in
that form using the relation $a_1b_1a_1^{-1}=b_1^{-1}$.
\end{remark}

Suppose we are given a $\zz G$-module $A$. The problem of calculating
the (co)homology groups $H_*(G;A)$ and $H^*(G;A)$ can be solved using
our free resolution of $\zz$ over $\zz G$. If $B$ is also a $\zz
G$-module, another problem is to determine the cup products
\[
H^1(G;A) \otimes H^1(G;B) \stackrel{\smile}{\to} H^2(G;A\otimes B)\phantom{.}
\]
and
\[
H^1(G;A) \otimes H^2(G;B) \stackrel{\smile}{\to} H^3(G;A\otimes B).
\]
As mentioned in the second remark after
Proposition~\ref{proposition:handeldiagonal}, the product $H^1(G;A)
\otimes H^1(G;B) \stackrel{\smile}{\to} H^2(G;A\otimes B)$ can be
computed using the map $\Delta_{1,1}$ of
Theorem~\ref{theorem:Gdiagonal}.

In order to compute $H^1(G;A) \otimes H^2(G;B) \stackrel{\smile}{\to}
H^3(G;A\otimes B)$, we recall~\cite[Chapter V and Section
  VIII.10]{Brown82} that there is a commutative diagram
\begin{equation}\label{eq:cupcap}
\begin{tikzcd}
H^1(G;A)\otimes H^2(G;B) \arrow{r}{\smile} \arrow{d}{\cong}[swap]{1\otimes(\underline{\phantom{M}}\frown z)} & H^3(G;A\otimes B) \arrow{d}{(\underline{\phantom{M}}\frown z)}[swap]{\cong} \\
H^1(G;A)\otimes H_1(G;B) \arrow{r}{\frown} & H_0(G;A \otimes B)
\end{tikzcd}
\end{equation}
where $z$ is a generator of $H_3(G;\zz)\cong \zz$. In terms of the
free resolution $F$, the cap product $H^1(G;A)\otimes H_1(G;B)
\stackrel{\frown}{\to} H_0(G;A\otimes B)$ can be calculated from a
diagonal approximation $\Delta\colon F \to (F\otimes F)$ by the
composition
\begin{equation}\label{eq:gamma}
\begin{tikzpicture}
\node(n01) at (0,1) [anchor=east] {$\Hom_G(F_1,A) \otimes (F_1\otimes_G B)$};
\node(n21) at (2,1) [anchor=west] {$\Hom_G(F_1,A)\otimes \bigl((F\otimes F)_1\otimes B\bigr)$};
\draw[->] (n01.east) -- (n21.west);
\node(n11) at (1,1) [anchor=south] {$\scriptstyle 1\otimes(\Delta\otimes 1)$};
\node(n00) at (0,0) [anchor=east] {};
\node(n20) at (2,0) [anchor=west] {$F_0\otimes (A\otimes B)$};
\draw[->] (n00.east) -- (n20.west);
\node(n10) at (1,0) [anchor=south] {$\scriptstyle \gamma$};
\end{tikzpicture}
\end{equation}
and then taking cohomology classes, where $\gamma$ is given by
$\gamma\bigl(u\otimes (x\otimes y\otimes n)\bigr) =
(-1)^{\deg(u)\deg(x)} x\otimes u(y)\otimes n$
(see~\cite[Section~V.3]{Brown82} for more details). All we need to
know about $\Delta$ to perform all the above computations is
$\Delta_{0,1}$, which is a summand of $\Delta_1$ given by
Theorem~\ref{theorem:Gdiagonal}.

Therefore, in order to completely understand the multiplicative
structure given by the cup product $H^*(G;A)\otimes H^*(G;B)
\stackrel{\smile}{\to} H^*(G;A\otimes B)$, we must understand the
isomorphism
\begin{equation}\label{eq:isoz}
\varphi_n = \underline{\phantom{M}}\frown z\colon H^n(G;A) \to H_{3-n}(G;A).
\end{equation}
Following~\cite[Section V.4]{Brown82}, $H_*(G;\zz) \cong H_*(F\otimes_G
F)$ and there is a product
\[
\Hom_G(F,A) \otimes (F\otimes_G F) \to F\otimes_G A
\]
given by
\begin{equation}\label{eq:capexplicit}
u\otimes (x\otimes y) \mapsto (-1)^{\deg(u)\deg(x)}x\otimes u(y)
\end{equation}
that induces the cap product 
\[
H^p(G;A)\otimes H_q(G;\zz) \to H_{q-p}(G;A)
\]
In particular, the isomorphism $\varphi_n$ can be
calculated if we find an element $\zeta \in (F\otimes_G F)_3$ such
that $[\zeta]=z \in H_3(G;\zz)$.

In the double complex $(F\otimes_G F)$, the differential $\partial$ is
given by
\begin{align*}
\partial_{p+q} \colon (F\otimes_G F)_{p+q} &\to (F\otimes_G F)_{p+q-1} \\
\partial_{p+q}(x\otimes y) &= \partial_{p,q}'(x\otimes y) +
\partial_{p,q}''(x\otimes y),
\end{align*}
where $\partial_{p,q}'(x\otimes y) = d_p(x)\otimes y$ and
$\partial_{p,q}''(x\otimes y) = (-1)^px\otimes d_q(y)$ for $x\in F_p$ and
$y\in F_q$.

\[
\begin{tikzcd}
F_3 \otimes_G F_3 \arrow{r} \arrow{d} & F_2 \otimes_G F_3 \arrow{r} \arrow{d} & F_1 \otimes_G F_3 \arrow{r}{\partial_{1,3}'} \arrow{d} & F_0 \otimes_G F_3 \arrow[twoheadrightarrow]{r}{\varepsilon_3'} \arrow{d}{\partial_{0,3}''} & \zz\phantom{^3}  \\
F_3 \otimes_G F_2 \arrow{r} \arrow{d} & F_2 \otimes_G F_2 \arrow{r} \arrow{d} & F_1 \otimes_G F_2 \arrow{r}{\partial_{1,2}'} \arrow{d} & F_0 \otimes_G F_2 \arrow[twoheadrightarrow]{r}{\varepsilon_2'} \arrow{d} & \zz^3  \\
F_3 \otimes_G F_1 \arrow{r} \arrow{d} & F_2 \otimes_G F_1 \arrow{r} \arrow{d} & F_1 \otimes_G F_1 \arrow{r} \arrow{d} & F_0 \otimes_G F_1 \arrow{d}  \\
F_3 \otimes_G F_0 \arrow{r} & F_2 \otimes_G F_0 \arrow{r} & F_1 \otimes_G F_0 \arrow{r} & F_0 \otimes_G F_0  \\
\end{tikzcd}
\]

Since $F_3=\zz G$, the cokernel of $\partial_{1,3}'$ is isomorphic to
$\zz$ via the map $\varepsilon'\colon F_0\otimes_G F_3 \to \zz$
defined by $\varepsilon'(x\otimes 1) = \varepsilon(x)$, where
$\varepsilon \colon \zz G\to \zz$ is the augmentation. Similarly,
$\coker(\partial_{1,2}') \cong \zz^3$, with the isomorphism
$\varepsilon_2'\colon F_0\otimes_G F_2 \to \zz^3$ defined in the
obvious way using $\varepsilon$.

The element $(1\otimes 1) \in F_0\otimes_G F_3$ is such that
$\varepsilon_3'(1\otimes 1) = 1 \ne 0$, hence $(1\otimes 1)\not\in
\im(\partial_{1,3}')$. On the other hand, the element $d_3(1) =
X_0\rho_1 + Y_0\rho_2 + Z_0\rho_3 \in F_2$ satisfies $\varepsilon(X_0)
= \varepsilon(Y_0) = \varepsilon(Z_0) = 0$, therefore
$\varepsilon_2'\partial_{0,3}''(1\otimes 1) = 0$, which means that
$\partial_{0,3}''(1\otimes 1) \in \im(\partial_{1,2}')$. A simple
diagram chasing now shows that there is $\zeta\in (F\otimes F)_3$ such
that $\pi_{0,3}(\zeta) = 1\otimes 1$ (where $\pi_{p,q}\colon (F\otimes
F)_{p+q}\to F_p\otimes F_q$ denotes the projection) and $\zeta \in
(\ker(\partial_3)-\im(\partial_4))$, hence $z = [\zeta] \in
H_3(G;\zz)$ is nonzero. If there is an integer $k$ and $\omega\in
(F\otimes F)_3$ such that $w = [\omega] \in H_3(G;\zz)$ satisifes $z =
kw$, then $\varepsilon_3'\pi_{0,3}(\zeta-k\omega) = 0 \iff 1 =
k\cdot\varepsilon_3'\pi_{0,3}(\omega) \then k = \pm 1$. Therefore $z$
generates $H_3(G;\zz)$ and is the same $z$ that gives the isomorphism
$\varphi_n$ of~(\ref{eq:isoz}). Now the formula~(\ref{eq:capexplicit})
tells us that, for the computation of $\varphi_3$, we need
$\pi_{0,3}(\zeta)$, which is simply $(1\otimes 1)$. To calculate
$\varphi_2$, the same formula~(\ref{eq:capexplicit}) requires
$\pi_{1,2}(\zeta)$. This can be determined in the following way: let
$\psi\colon \zz G\to \zz G$ denote the function $\psi\left(\sum
\alpha_gg\right) = \sum \alpha_gg^{-1}$. We have
\begin{align*}
\partial_{0,3}''(1\otimes 1) &= 1\otimes X_0\rho_1 + 1\otimes
Y_0\rho_2 + 1\otimes Z_0\rho_3 \\
&= \psi(X_0)\otimes \rho_1 + \psi(Y_0)\otimes \rho_2 +
\psi(Z_0)\otimes \rho_3
\end{align*}
In the group ring $\zz G$, the identities
\begin{align*}
(gh-1) &= g(h-1) + (g-1), \\
(g^n-1) &= \dfrac{\partial g^n}{\partial g}(g-1)
\end{align*}
for $g$, $h\in G$ allow us to write each of the elements $\psi(X_0)$,
$\psi(Y_0)$ and $\psi(Z_0)$ as
\begin{align*}
\psi(X_0) &= X_{a_1}(a_1-1) + X_{b_1}(b_1-1) + X_{a_2}(a_2-1), \\
\psi(Y_0) &= Y_{a_1}(a_1-1) + Y_{b_1}(b_1-1) + Y_{a_2}(a_2-1), \\
\psi(Z_0) &= Z_{a_1}(a_1-1) + Z_{b_1}(b_1-1) + Z_{a_2}(a_2-1),
\end{align*}
and then
\begin{align*}
\psi(X_0)\otimes \rho_1 + \psi(Y_0)\otimes \rho_2 +
\psi(Z_0)\otimes \rho_3 &= \partial_{1,2}'\bigl((X_{a_1}\alpha_1 + X_{b_1}\beta_1 + X_{a_2}\alpha_2)\otimes \rho_1 +\\
&\phantom{= \partial_{1,2}'\bigl(}(Y_{a_1}\alpha_1 + Y_{b_1}\beta_1 + Y_{a_2}\alpha_2)\otimes \rho_2 +\\
&\phantom{= \partial_{1,2}'\bigl(}(Z_{a_1}\alpha_1 + Z_{b_1}\beta_1 + Z_{a_2}\alpha_2)\otimes \rho_3\bigr).
\end{align*}
This means we can take
\begin{align}\label{eq:pi12}
\pi_{1,2}(\zeta) &= (X_{a_1}\alpha_1 + X_{b_1}\beta_1 + X_{a_2}\alpha_2)\otimes \rho_1 + \notag\\
&\phantom{= (}(Y_{a_1}\alpha_1 + Y_{b_1}\beta_1 + Y_{a_2}\alpha_2)\otimes \rho_2 + \\
&\phantom{= (}(Z_{a_1}\alpha_1 + Z_{b_1}\beta_1 + Z_{a_2}\alpha_2)\otimes \rho_3. \notag
\end{align}

\section{Examples of computations}\label{section:examples}

In this section we compute the cohomology groups $H^*(G;\tilde{\zz})$,
where $\tilde{\zz}$ represents a non-trivial $G$-module with
underlying abelian group $\zz$, show some examples of computation of
the cup products for these twisted integer coefficients, and also
determine the cohomology ring $H^*(G;\zz_p)$ for an odd prime $p$.

\bigskip
First we determine the groups $H^*(G;\zz_\eta)$, where $\zz_\eta$
stands for the $G$-module $\zz$ determined by the homomorphism
$\eta\colon G \to \Aut(\zz) = \{1,-1\}$. We begin the analysis with
the maps $\eta$ such that $\eta(b_1) = -1$ (notice that this implies
that $s$ is even).

\begin{theorem}
Suppose that $\eta\colon G\to \Aut(\zz)$ is such that $\eta(b_1) =
-1$. Then
\begin{align*}
H^0(G;\zz_\eta) &= 0, \\
H^1(G;\zz_\eta) &= \zz_2, \\
H^2(G;\zz_\eta) &= \zz_2\oplus\zz_2, \\
H^3(G;\zz_\eta) &= \zz_2.
\end{align*}
\end{theorem}
\begin{proof}
We could compute the cohomology groups using the projective resolution
of Theorem~\ref{theorem:Gresolution}, but a spectral sequence argument
readily computes them when the action $\eta$ is such that $\eta(b_1) =
-1$. Let $N$ and $Q$ be the subgroups of $G$ mentioned in
Section~\ref{section:G}. As a $N$-module, $\zz_\eta$ is such that
$x\cdot k = k$ and $y\cdot k = -k$ for all $k\in\zz_\eta$. A direct
calculation using, for example, the free resolution of $\zz$ over $\zz
N$ found in~\cite{MartinsGoncalves} implies $H^0(N;\zz_\eta) = 0$,
$H^1(N;\zz_\eta) = \zz_2$ and $H^2(N;\zz_\eta) = \zz_2$. Hence the
Lyndon-Hochschild-Serre spectral sequence associated with the
extension
\[
\begin{tikzcd}
N \arrow[tail]{r} & G \arrow[twoheadrightarrow]{r} & Q
\end{tikzcd}
\]
is such that $E_2^{p,q} = H^p(Q; H^q(N;\zz_\eta))$ is zero for $q=0$
and, if $q\in \{1,2\}$,
\[
E_2^{p,q} = \begin{cases}
\zz_2, & \text{if $p=0$}, \\
\zz_2\oplus\zz_2, & \text{if $p\ge 1$}.
\end{cases}
\]
Let's plot the lines corresponding to $q=1$ and $q=2$ of the term
$E_2$:
\[
\begin{tikzcd}[column sep=1.5em]
(q=2) & \zz_2 \arrow[hook]{drr}[swap]{d_2^{0,2}}& \zz_2\oplus\zz_2  \arrow{drr}[description]{\cong} & \zz_2\oplus\zz_2 \arrow{drr}[description]{\cong} & \zz_2\oplus\zz_2 & \zz_2\oplus\zz_2 & \cdots\\
(q=1) & \zz_2 & \zz_2\oplus\zz_2 & \zz_2\oplus\zz_2 & \zz_2\oplus\zz_2 & \zz_2\oplus\zz_2 & \cdots
\end{tikzcd}
\]
Since $E_3 = E_\infty$ and $H^n(G;\zz_\eta) = 0$ for $n\ge 4$, the
maps $d_2^{p,2}\colon E_2^{p,2} \to E_2^{p+2,1}$ must be isomorphisms
for $p\ge 1$. Also, using Poincar\'e duality, we see that
$H^3(G;\zz_\eta) = H_0(G;\zz_\eta) \cong \zz_2$, since the action of
$b_1\in G$ on $\zz_\eta$ is non-trivial. Therefore the map $d_2^{0,2}$
must be injective and the theorem follows.
\end{proof}

The actions $\eta\colon G\to \Aut(\zz)$ that remain satisfy $\eta(b_1)
= 1$. Hence, as we've already computed the cohomology groups of $G$
with trivial $\zz$ coefficients, we are left with three actions to
consider. For all non-trivial actions $\eta\colon G\to \Aut(\zz)$, we
have $H^0(G;\zz_\eta) = 0$ and $H^3(G;\zz_\eta) \cong H_0(G;\zz_\eta)
\cong \zz_2$. Also, the groups $H^1(G;\zz_\eta)$ and $H^2(G;\zz_\eta)$
can be readily computed from the resolution of
Theorem~\ref{theorem:Gresolution} and we get the following result:

\begin{theorem}\label{theorem:cohomologygroupsb11}
If $\eta_1\colon G\to \Aut(\zz)$ is the action given by $\eta_1(a_1) =
\eta_1(b_1) = 1$, $\eta_1(a_2) = -1$, we get $H^1(G;\zz_{\eta_1}) = \zz_2$
and
\[
H^2(G;\zz_{\eta_1}) = \begin{cases}
\zz_{2r}\oplus\zz_2, & \text{is $s$ is even},\\
\zz_{4r}, & \text{if $s$ is odd}.
\end{cases}
\]
For the action $\eta_2\colon G\to \Aut(\zz)$ defined by $\eta_2(a_1) =
-1$, $\eta_2(b_1) = \eta_2(a_2) = 1$, we get $H^1(G;\zz_{\eta_2}) = \zz_2$
and
\[
H^2(G;\zz_{\eta_2}) = \begin{cases}
\zz_{2u}\oplus\zz_2, & \text{is $s$ is even},\\
\zz_{4u}, & \text{if $s$ is odd}.
\end{cases}
\]
Finally, if $\eta_3\colon G\to \Aut(\zz)$ is the action given by
$\eta_3(a_1) = -1$, $\eta_3(b_1) =1$, $\eta_3(a_2) = -1$, we get
$H^1(G;\zz_{\eta_3}) = \zz\oplus\zz_2$ and $H^2(G;\zz_{\eta_3}) =
\zz\oplus\zz_s$.
\end{theorem}

\begin{example}
Let's calculate the cup products
\[
H^p(G;\zz_{\eta_1})\otimes H^q(G;\zz_{\eta_2}) \stackrel{\smile}{\to} H^{p+q}(G;\zz_{\eta_1}\otimes\zz_{\eta_2}) \cong H^{p+q}(G;\zz_{\eta_3}).
\]
Since the cohomology groups depend on the parity of $s$, let's assume
that $s$ is odd. In this case, we have $H^1(G;\zz_{\eta_1})\cong
H^1(G;\zz_{\eta_2}) \cong \zz_2$ and $H^2(G;\zz_{\eta_3}) =
\zz\oplus\zz_s$, so the product
\[
H^1(G;\zz_{\eta_1})\otimes H^1(G;\zz_{\eta_2}) \stackrel{\smile}{\to} H^2(G;\zz_{\eta_3})
\]
is zero. In order to compute $H^1(G;\zz_{\eta_1})\otimes
H^2(G;\zz_{\eta_2}) \stackrel{\smile}{\to} H^3(G;\zz_{\eta_3})$, we
use~\eqref{eq:cupcap} with $A=\zz_{\eta_1}$ and $B=\zz_{\eta_2}$:
\[
\begin{tikzcd}
H^1(G;\zz_{\eta_1})\otimes H^2(G;\zz_{\eta_2}) \arrow{r}{\smile} \arrow{d}{\cong}[swap]{1\otimes(\underline{\phantom{M}}\frown z)} & H^3(G;\zz_{\eta_3}) \arrow{d}{\underline{\phantom{M}}\frown z}[swap]{\cong} \\
H^1(G;\zz_{\eta_1})\otimes H_1(G;\zz_{\eta_2}) \arrow{r}{\frown} & H_0(G;\zz_{\eta_3})
\end{tikzcd}
\]

One generator of $H^1(G;\zz_{\eta_1}) \cong \zz_2$ is the class of the
map $\alpha_2^*\in \Hom_G(F_1,\zz_{\eta_1})$ defined by
$\alpha_2^*(\alpha_1) = 0$, $\alpha_2^*\beta_1) = 0$ and
$\alpha_2^*(\alpha_2) = 1$, and one generator of $H_1(G;\zz_{\eta_2})
\cong \zz_{4u}$ is the class of $(\beta_1\otimes
1)$. Using~\eqref{eq:gamma}, we get
\[
[\varphi]\frown[\beta_1\otimes 1] = 0,
\]
hence the product $H^1(G;\zz_{\eta_1})\otimes H^2(G;\zz_{\eta_2})
\stackrel{\smile}{\to} H^3(G;\zz_{\eta_3})$ is also null.
\end{example}

\begin{example}
Assuming $s$ odd, let's compute the cup products
\[
H^p(G;\zz_{\eta_3})\otimes H^q(G;\zz_{\eta_3}) \stackrel{\smile}{\to}
H^{p+q}(G;\zz_{\eta_3}\otimes\zz_{\eta_3}) \cong H^{p+q}(G;\zz).
\]
In order to do that, we need representatives for the generating
classes of the cohomology groups $H^1(G;\zz_{\eta_3})$,
$H^2(G;\zz_{\eta_3})$ and $H^2(G;\zz)$, which can be obtained using
the free resolution $F$. The group $H^1(G;\zz_{\eta_3})\cong
\zz\oplus\zz_2$ is generated by $[\alpha_2^*]$ and
$[\alpha_1^*+\alpha_2^*]$, where $[\alpha_2^*]$ generates a subgroup
isomorphic to $\zz$ and $[\alpha_1^*+\alpha_2^*]$ generates a subgroup
isomorphic to $\zz_2$. The group $H^2(G;\zz_{\eta_3})\cong
\zz\oplus\zz_s$ is generated by $[\rho_1^*+\rho_3^*]$ and
$[\rho_2^*]$, where $[\rho_1^*+\rho_3^*]$ generates a subgroup
isomorphic to $\zz$ and $[\rho_2^*]$ generates a subgroup isomorphic
to $\zz_s$. The group $H^2(G;\zz) = \zz_{4t}\oplus\zz_4$ is generated
by $[\rho_3^*]$ and $[\rho_1^*+u\rho_3^*]$, where $[\rho_3^*]$
generates a subgroup isomorphic to $\zz_{4t}$ and
$[\rho_1^*+u\rho_3^*]$ generates a subgroup isomorphic to $\zz_4$. We
also note that, in $H^2(G;\zz)$, $[2\rho_1^*+\rho_2^*+2u\rho_3^*] =
0$.

Using the map $\Delta_{1,1}$ of Theorem~\ref{theorem:Gdiagonal}, we
determine the cup products $H^1(G;\zz_{\eta_3})\otimes
H^1(G;\zz_{\eta_3}) \stackrel{\smile}{\to} H^2(G;\zz)$:
\begin{align*}
&[\alpha_2^*]^2 = 2[\rho_1^*+u\rho_3^*], \\
&[\alpha_2^*]\smile[\alpha_1^*+\alpha_2^*] = 2t[\rho_3^*]+2[\rho_1^*+u\rho_3^*], \\
&[\alpha_1^*+\alpha_2^*]^2 = 2t[\rho_3^*]-2(r-1)[\rho_1^*+u\rho_3^*].
\end{align*}

Now let's calculate the products $H^1(G;\zz_{\eta_3})\otimes
H^2(G;\zz_{\eta_3}) \stackrel{\smile}{\to} H^3(G;\zz)$. Since
$H^3(G;\zz) = \zz$ is torsion-free, we have
\begin{align*}
&[\alpha_1^*+\alpha_2^*]\smile [\rho_1^*+\rho_3^*] = 0,\\
&[\alpha_1^*+\alpha_2^*]\smile [\rho_2^*] = 0,\\
&[\alpha_2^*]\smile [\rho_2^*] = 0.\\
\end{align*}
It remains to compute $[\alpha_2^*]\smile[\rho_1^*+\rho_3^*]$. The
isomorphism $\varphi_2$ of~\eqref{eq:isoz} maps $[\rho_1^*+\rho_3^*]$ to
\[
\xi=[(X_{a_1}\alpha_1 + X_{b_1}\beta_1 + X_{a_2}\alpha_2)\otimes 1 + (Z_{a_1}\alpha_1 + Z_{b_1}\beta_1 + Z_{a_2}\alpha_2)\otimes 1],
\]
according to~\eqref{eq:pi12}. But the expression for $X_0$
of~\eqref{eq:X0} shows that $X_{a_2}=0$. We also have
\begin{align*}
\psi(Z_0) &= (a_2^{-1}-1) - (x^{-t}y^{-u}-1) \\
&= -\frac{\partial a_1^{-2t}}{\partial a_1}(a_1-1) -a_1^{-2t}\frac{\partial b_1^{-u}}{\partial b_1}(b_1-1) - a_2^{-1}(a_2-1),
\end{align*}
hence $Z_{a_2} = -a_2^{-1}$. Now~\eqref{eq:capexplicit} shows that
$[\alpha_2^*]\frown \xi$ is a generator of $H_0(G;\zz)$, hence
$[\alpha_2^*]\smile [\rho_1^*+\rho_3^*]$ generates $H^3(G;\zz)$.
\end{example}

Now we determine $H^*(G;\zz_p)$ when $p$ is an odd prime. The case
$p=2$ has been solved by Hillman~\cite{Hillman}, but could also be
recovered from our techniques.

\begin{theorem}\label{theorem:modpcohomology}
The cohomology groups $H^*(G;\zz_p)$, for an odd prime $p$, are given
by
\begin{align*}
H^0(G;\zz_p) &\cong \zz_p, \\
H^1(G;\zz_p) &\cong \begin{cases}
\zz_p, & \text{ if $p\mid t$}, \\
0, & \text{ if $p\nmid t$},
\end{cases} \\
H^2(G;\zz_p) &\cong \begin{cases}
\zz_p, & \text{ if $p\mid t$}, \\
0, & \text{ if $p\nmid t$},
\end{cases} \\
H^3(G;\zz_p) &\cong \zz_p.
\end{align*}
Moreover, if $p\mid t$, there is a generator $\alpha\in H^1(G;\zz_p)$
and a generator $\beta\in H^2(G;\zz_p)$ such that $\alpha\smile\beta$
is a generator of $H^3(G;\zz_p)$, so
\[
H^*(G;\zz_p) \cong \frac{\zz_p[\alpha,\beta]}{(\alpha^2=0, \beta^2=0)},
\]
where $\dim(\alpha)=1$ and $\dim(\beta)=2$.
\end{theorem}

\begin{proof}
The computation of $H^0(G;\zz_p)$ is immediate.  Using
Theorem~\ref{theorem:integralcohomology} and the universal coefficient
theorems for homology and cohomology, we can compute the other groups
$H^*(G;\zz_p)$: from the short exact sequence
\[
\begin{tikzcd}[column sep=1.5em]
0 \arrow{r} & H_1(G;\zz)\otimes \zz_p \arrow{r} & H_1(G;\zz_p) \arrow{r} & \underbrace{\Tor(\zz,\zz_p)}_{=0} \arrow{r} & 0,
\end{tikzcd}
\]
we get $H^2(G;\zz_p) \cong H_1(G;\zz_p) \cong H_1(G;\zz)\otimes \zz_p
\cong \zz_{\gcd(p,t)}$ by Poincar\'e duality, and from
\[
\begin{tikzcd}[column sep=1.5em]
0 \arrow{r} & \underbrace{\Ext(H_0(G;\zz),\zz_p)}_{=0} \arrow{r} & H^1(G;\zz_p) \arrow{r} & \Hom(H_1(G;\zz),\zz_p) \arrow{r} & 0
\end{tikzcd}
\]
we get $H^1(G;\zz_p) \cong \zz_{\gcd(p,t)}$. Finally, Poincar\'e
duality also gives us $H^3(G;\zz_p) \cong H_0(G;\zz_p) \cong
\zz_p$. Thus, if $p\nmid t$, there is no cup product to consider and
we are done.

Suppose then that $p\mid t$. From Theorem~\ref{theorem:Gresolution},
it is easy to check that $v=[\alpha_1^*+r\alpha_2^*]$ generates
$H^1(G;\zz_p)\cong\zz_p$ and $w=[\alpha_2\otimes 1]$ generates
$H_1(G;\zz_p)\cong\zz_p$. But now~\eqref{eq:gamma} and the fact that
$p\nmid r$ imply that $v\frown w$ generates $H_0(G;\zz_p)$. We get the
statement of the theorem defining $\alpha=v$, $\beta =
\varphi_2^{-1}(w)$ and observing that $\alpha^2=0$ since $p$ is odd.
\end{proof}

{\small
\bibliographystyle{amsalpha}
\bibliography{sapphires}
%\nocite{*}
}

\end{document}